\documentclass[12pt,leqno,draft]{article}
\usepackage[pagewise]{lineno}
\usepackage{amsfonts}
\usepackage{amsmath, amsthm, amsfonts, amssymb, color}
\usepackage{mathrsfs}
\usepackage{color}
\usepackage{stmaryrd}
\newtheorem{thm}{Theorem}[section]

\newtheorem{lem}[thm]{Lemma}

\theoremstyle{definition}
\newtheorem{defn}{Definition}[section]
\newcommand{\scr}[1]{\mathscr #1}
\definecolor{wco}{rgb}{0.5,0.2,0.3}

\numberwithin{equation}{section} \theoremstyle{remark}

\newcommand{\ua}{\uparrow}

\title{{\bf Harnack Inequality for Distribution Dependent Stochastic Hamiltonian System}\footnote{Supported in
 part by  NNSFC (12271398).}
}
\author{
{\bf   Xing Huang $^{a)}$, Xiaochen Ma $^{a)}$    }\\
\footnotesize{ a)Center for Applied Mathematics, Tianjin
University, Tianjin 300072, China}\\
\footnotesize{  xinghuang@tju.edu.cn, }
\footnotesize{   maxiaochen@tju.edu.cn}}
\begin{document}
\allowdisplaybreaks
\def\R{\mathbb R}  \def\ff{\frac} \def\ss{\sqrt} \def\B{\mathbf
B}
\def\N{\mathbb N} \def\kk{\kappa} \def\m{{\bf m}}
\def\ee{\varepsilon}\def\ddd{D^*}
\def\dd{\delta} \def\DD{\Delta} \def\vv{\varepsilon} \def\rr{\rho}
\def\<{\langle} \def\>{\rangle} \def\GG{\Gamma} \def\gg{\gamma}
  \def\nn{\nabla} \def\pp{\partial} \def\E{\mathbb E}
\def\d{\text{\rm{d}}} \def\bb{\beta} \def\aa{\alpha} \def\D{\scr D}
  \def\si{\sigma} \def\ess{\text{\rm{ess}}}
\def\beg{\begin} \def\beq{\begin{equation}}  \def\F{\scr F}
\def\Ric{\text{\rm{Ric}}} \def\Hess{\text{\rm{Hess}}}
\def\e{\text{\rm{e}}} \def\ua{\underline a} \def\OO{\Omega}  \def\oo{\omega}
 \def\tt{\tilde} \def\Ric{\text{\rm{Ric}}}
\def\cut{\text{\rm{cut}}} \def\P{\mathbb P} \def\ifn{I_n(f^{\bigotimes n})}
\def\C{\scr C}   \def\G{\scr G}   \def\aaa{\mathbf{r}}     \def\r{r}
\def\gap{\text{\rm{gap}}} \def\prr{\pi_{{\bf m},\varrho}}  \def\r{\mathbf r}
\def\Z{\mathbb Z} \def\vrr{\varrho} \def\ll{\lambda}
\def\L{\scr L}\def\Tt{\tt} \def\TT{\tt}\def\II{\mathbb I}
\def\i{{\rm in}}\def\Sect{{\rm Sect}}  \def\H{\mathbb H}
\def\M{\scr M}\def\Q{\mathbb Q} \def\texto{\text{o}} \def\LL{\Lambda}
\def\Rank{{\rm Rank}} \def\B{\scr B} \def\i{{\rm i}} \def\HR{\hat{\R}^d}
\def\to{\rightarrow}\def\l{\ell}\def\iint{\int}
\def\EE{\scr E}\def\no{\nonumber}
\def\A{\scr A}\def\V{\mathbb V}\def\osc{{\rm osc}}
\def\BB{\scr B}\def\Ent{{\rm Ent}}\def\3{\triangle}\def\H{\scr H}
\def\U{\scr U}\def\8{\infty}\def\1{\lesssim}\def\HH{\mathrm{H}}
 \def\T{\scr T}
 \def\R{\mathbb R}  \def\ff{\frac} \def\ss{\sqrt} \def\B{\mathbf
B} \def\W{\mathbb W}
\def\N{\mathbb N} \def\kk{\kappa} \def\m{{\bf m}}
\def\ee{\varepsilon}\def\ddd{D^*}
\def\dd{\delta} \def\DD{\Delta} \def\vv{\varepsilon} \def\rr{\rho}
\def\<{\langle} \def\>{\rangle} \def\GG{\Gamma} \def\gg{\gamma}
  \def\nn{\nabla} \def\pp{\partial} \def\E{\mathbb E}
\def\d{\text{\rm{d}}} \def\bb{\beta} \def\aa{\alpha} \def\D{\scr D}
  \def\si{\sigma} \def\ess{\text{\rm{ess}}}
\def\beg{\begin} \def\beq{\begin{equation}}  \def\F{\scr F}
\def\Ric{\text{\rm{Ric}}} \def\Hess{\text{\rm{Hess}}}
\def\e{\text{\rm{e}}} \def\ua{\underline a} \def\OO{\Omega}  \def\oo{\omega}
 \def\tt{\tilde} \def\Ric{\text{\rm{Ric}}}
\def\cut{\text{\rm{cut}}} \def\P{\mathbb P} \def\ifn{I_n(f^{\bigotimes n})}
\def\C{\scr C}      \def\aaa{\mathbf{r}}     \def\r{r}
\def\gap{\text{\rm{gap}}} \def\prr{\pi_{{\bf m},\varrho}}  \def\r{\mathbf r}
\def\Z{\mathbb Z} \def\vrr{\varrho} \def\ll{\lambda}
\def\L{\scr L}\def\Tt{\tt} \def\TT{\tt}\def\II{\mathbb I}
\def\i{{\rm in}}\def\Sect{{\rm Sect}}  \def\H{\mathbb H}
\def\M{\scr M}\def\Q{\mathbb Q} \def\texto{\text{o}} \def\LL{\Lambda}
\def\Rank{{\rm Rank}} \def\B{\scr B} \def\i{{\rm i}} \def\HR{\hat{\R}^d}
\def\to{\rightarrow}\def\l{\ell}\def\iint{\int}
\def\EE{\scr E}\def\Cut{{\rm Cut}}
\def\A{\scr A} \def\Lip{{\rm Lip}}
\def\BB{\scr B}\def\Ent{{\rm Ent}}\def\L{\scr L}
\def\R{\mathbb R}  \def\ff{\frac} \def\ss{\sqrt} \def\B{\mathbf
B}
\def\N{\mathbb N} \def\kk{\kappa} \def\m{{\bf m}}
\def\dd{\delta} \def\DD{\Delta} \def\vv{\varepsilon} \def\rr{\rho}
\def\<{\langle} \def\>{\rangle} \def\GG{\Gamma} \def\gg{\gamma}
  \def\nn{\nabla} \def\pp{\partial} \def\E{\mathbb E}
\def\d{\text{\rm{d}}} \def\bb{\beta} \def\aa{\alpha} \def\D{\scr D}
  \def\si{\sigma} \def\ess{\text{\rm{ess}}}
\def\beg{\begin} \def\beq{\begin{equation}}  \def\F{\scr F}
\def\Ric{\text{\rm{Ric}}} \def\Hess{\text{\rm{Hess}}}
\def\e{\text{\rm{e}}} \def\ua{\underline a} \def\OO{\Omega}  \def\oo{\omega}
 \def\tt{\tilde} \def\Ric{\text{\rm{Ric}}}
\def\cut{\text{\rm{cut}}} \def\P{\mathbb P} \def\ifn{I_n(f^{\bigotimes n})}
\def\C{\scr C}      \def\aaa{\mathbf{r}}     \def\r{r}
\def\gap{\text{\rm{gap}}} \def\prr{\pi_{{\bf m},\varrho}}  \def\r{\mathbf r}
\def\Z{\mathbb Z} \def\vrr{\varrho} \def\ll{\lambda}
\def\L{\scr L}\def\Tt{\tt} \def\TT{\tt}\def\II{\mathbb I}
\def\i{{\rm in}}\def\Sect{{\rm Sect}}  \def\H{\mathbb H}
\def\M{\scr M}\def\Q{\mathbb Q} \def\texto{\text{o}} \def\LL{\Lambda}
\def\Rank{{\rm Rank}} \def\B{\scr B} \def\i{{\rm i}} \def\HR{\hat{\R}^d}
\def\to{\rightarrow}\def\l{\ell}
\def\8{\infty}\def\I{1}\def\U{\scr U}
\maketitle

\begin{abstract} The dimension free Harnack inequality is established for the distribution dependent stochastic Hamiltonian system, where the drift is Lipschitz continuous in the measure variable under the distance induced by the H\"{o}lder-Dini continuous functions, which are $\beta (\beta>\frac{2}{3})$-H\"{o}lder continuous on the degenerate component and square root of Dini continuous on the non-degenerate one. The results extend the existing ones in which the drift is Lipschitz continuous in the measure variable under $L^2$-Wasserstein distance.

\end{abstract} \noindent
 \vskip 2cm

\section{Introduction}
The stochastic Hamiltonian system,  including the kinetic
Fokker-Planck equation (see \cite{V}), is a classical degenerate model. In \cite{FF} the authors study the regularity of stochastic kinetic equations. \cite{GW} investigates the Bismut formula, gradient estimate and Harnack inequality by using the method of coupling by change of measure. \cite{WZ1} and \cite{Z1} focus on the derivative formula. \cite{W2} proves the the hypercontractivity. One can refer to the references in the papers mentioned above for more related results.

On the other hand, the McKean-Vlasov stochastic differential equations (SDEs), presented in \cite{M}, can be used to characterize the nonlinear Fokker-Planck-Kolmogorov equations. Recently, there are plentiful results on McKean-Vlasov SDEs, such as the well-posedness, Harnack inequality, the Bismut formula, exponential ergodicity, estimate of heat kernel, see \cite{C,CN,HS,HW18,HWJMAA,MV,RW,RZ,FYW1,FYW2,Z5} and references therein for more details. For the well-posedness, the drifts can be Lipschitz continuous in the measure variable under weighted variation distance, for instance \cite{HWJMAA,RZ,FYW2} and so on.  However, with respect to the Harnack inequality, most results concentrate on the case that the coefficients are Lipschitz continuous in the measure variable under $L^2$-Wasserstein distance, see \cite{RW} for the distribution dependent stochastic Hamiltonian system. In fact, for the well-posedness, the initial distributions (initial values) are assumed to be the same, while the Harnack inequality investigates the regularity of the nonlinear semigroup from different initial distributions, which will produce more difficulty than the study of well-posedness.

Quite recently, the first author and his co-author have established the log-Harnack inequality and Bismut derivative formula for non-degenerate McKean-Vlasov SDEs in \cite{HW22a}, where for the log-Harnack inequality, the drifts are only assumed to be Lipschitz continuous under the distance induced by square root of Dini continuous functions, which allows the drifts even being not Dini continuous in the $L^2$-Wasserstein distance.

In this paper, we intend to study the Harnack inequality for distribution dependent stochastic Hamiltonian system, where the drift is Lipschitz continuous in the measure variable under the distance induced by the H\"{o}lder-Dini continuous functions. More precisely, the functions are assumed to be $\beta (\beta>\frac{2}{3})$-H\"{o}lder continuous on the degenerate component and square root of Dini continuous on the non-degenerate one. Compared with \cite{HW22a}, we need to establish the gradient estimate of $P_t^\mu f$ with measure-valued curve parameter $\mu$ in non-degenerate and degenerate components respectively. Moreover, when $f$ only depends on the non-degenerate component, the gradient estimate on the degenerate component of $P_t^\mu f$ is also derived, which is crucial in the proof of the main result, see Theorem \ref{T3.2}(2) below.


Fix $m,d\in\mathbb{N}^{+}$.
Let $\scr P$ be the set of all probability measures in $\R^{m+d}$ equipped with the weak topology.
For $k\geq 1$, define
$$\scr P_k:=\big\{\mu\in \scr P: \|\mu\|_k:=\mu(|\cdot|^k)^{\ff 1 k}<\infty\big\}.$$
$\scr P_k$    is a Polish space under the $L^k$-Wasserstein distance
$$\W_k(\mu,\nu)= \inf_{\pi\in \C(\mu,\nu)} \bigg(\int_{\R^{m+d}\times\R^{m+d}} |x-y|^k \pi(\d x,\d y)\bigg)^{\ff 1 {k}},\ \  $$ where $\C(\mu,\nu)$ is the set of all couplings of $\mu$ and $\nu$.

For any $x\in\R^{m+d}$, let $x^{(1)}$ denote the first $m$ components and $x^{(2)}$ denote the last $d$ components, that is $x=(x^{(1)}, x^{(2)})\in\R^{m+d}$ with $x^{(1)}\in\R^m$ and $x^{(2)}\in\R^d$. Fix $T>0$. Consider the following distribution dependent stochastic Hamiltonian system on $\mathbb{R}^{m+d}$:
\beq\label{E0}
\begin{cases}
\d X_t^{(1)}=M X_t^{(2)}\d t, \\
\d X_t^{(2)}=B_t(X_t,\L_{X_t})\d t+\sigma_t\d W_t,
\end{cases}
\end{equation}
where $W=(W_t)_{t\geq 0}$ is a $d$-dimensional standard Brownian motion with respect to a complete filtration probability space $(\OO, \F, \{\F_{t}\}_{t\ge 0}, \P)$, $M$ is an $m\times d$ matrix, and $\sigma:[0,T]\to \mathbb{R}^{d}\otimes\R^d$, $B:[0,T]\times\R^{m+d}\times \scr P\to\mathbb{R}^d$ are measurable.

Recall that for two probability measures $\mu,\nu\in\scr P$, the entropy and total variation norm are defined as follows:
$$\Ent(\nu|\mu):= \beg{cases} \int_{\R^{m+d}} (\log \ff{\d\nu}{\d\mu})\d\nu, \ &\text{if}\ \nu\ \text{ is\ absolutely\ continuous\ with\ respect\ to}\ \mu,\\
 \infty,\ &\text{otherwise;}\end{cases}$$ and
$$\|\mu-\nu\|_{var} := \sup_{|f|\leq 1}|\mu(f)-\nu(f)|.$$ By Pinsker's inequality (see \cite{Pin}),
\beq\label{ETX} \|\mu-\nu\|_{var}^2\le 2 \Ent(\nu|\mu),\ \ \mu,\nu\in \scr P.\end{equation}

Throughout the paper, we will use $C$ or $c$ as a constant, the values of which may change from one place to another. For a function $f$ on $\R^{m+d}$ and $i=1,2$, let $\nabla^{(i)}f(x)$ stand for the gradient with respect to $x^{(i)}$.

The paper is organized as follows: In Section 2, we state the main results, i.e. the Harnack inequality for distribution dependent stochastic Hamiltonian system and the proof is provided in Section 3; In Section 4, the well-posedness for degenerate McKean-Vlasov SDEs is investigated, where the drifts are assumed to be Lipschitz continuous in the measure variable under the weighted variation distance plus the $L^k$-Wasserstein distance.
\section{Main Results}
Let
$$  \scr A:= \bigg\{\aa: [0,\infty)\to  [0,\infty)\  \text{is\ increasing\ and\ concave,\ } \aa(0)=0,\ \int_0^1 \ff{\aa(r)^2}{r}\d r\in (0,\infty)\bigg\}.$$
For $\beta\in(0,1]$, $\alpha\in\scr A$, define
$$\rho_{\beta,\alpha}(x,y)=|x^{(1)}-y^{(1)}|^\beta+\alpha(|x^{(2)}-y^{(2)}|),\ \ x,y\in\R^{m+d}.$$
 Since $\alpha$ is concave and increasing, we conclude that $\rho_{\beta,\alpha}$ is a distance on $\R^{m+d}$. For a real valued function $f$ on $\R^{m+d}$, let
$$[f]_{\beta,\alpha}:=\sup_{x\neq y} \ff{|f(x)-f(y)|}{\rho_{\beta,\alpha}(x,y)}.$$
Let $$\scr P_{\beta,\aa}:=\big\{\mu\in\scr P:\int_{\R^{m+d}}(|x^{(1)}|^\beta+\aa(|x^{(2)}|))\mu(\d x)<\infty\big\}.$$
Define the Wasserstein distance induced by $\rho_{\beta,\alpha}$:
$$\W_{\beta,\alpha}(\mu,\nu):= \inf_{\pi\in\C(\mu,\nu)}\int_{\R^{m+d}\times \R^{m+d}}\rho_{\beta,\alpha}(x,y)\pi(\d x,\d y),$$
and $\W_{\beta,\alpha}$ is a complete distance on $\scr P_{\beta,\aa}$.
Moreover, we have the dual formula
$$\W_{\beta ,\alpha}(\mu,\nu):= \sup_{[f]_{\beta,\alpha}\leq 1}|\mu(f)-\nu(f)|,\ \ \mu,\nu\in \scr P_{\beta,\aa}.$$
Noting that for any $\mu,\nu\in\scr P_{\beta,\alpha}$, $\{f: f\in\scr B_b(\R^{m+d}), [f]_{\beta,\alpha}\leq 1\}$ is dense in $\{f: [f]_{\beta,\alpha}\leq 1\}$ under $L^1(\mu+\nu)$, we have
$$\W_{\beta ,\alpha}(\mu,\nu):= \sup_{f\in\scr B_b(\R^{m+d}),[f]_{\beta,\alpha}\leq 1}|\mu(f)-\nu(f)|,\ \ \mu,\nu\in \scr P_{\beta,\aa}.$$
Furthermore, it follows from the concavity of $\aa$ and $\alpha(0)=0$ that
\beq\label{AA}\begin{split} 
\aa(rt)\le r\aa(t),\ \ t> 0, r\ge 1,
\end{split}\end{equation}
By \eqref{AA} for $t=1$ and $\alpha(t)\leq \alpha(1), t\in[0,1]$, we conclude that
\begin{align}\label{DDY}\alpha(r)\leq \alpha(1)(1+r),\ \ r\geq 0.\end{align}
So, for any $k\geq 1$,
\begin{equation}\label{kvs}\sup_{[f]_{\beta,\aa}\le 1} |f(x)-f(0)|\le |x^{(1)}|^\beta+\aa(|x^{(2)}|)\le 2(\alpha(1)+1)(1+|x|^k),\ \ x\in\R^{m+d}.
\end{equation}
Therefore $\scr P_k\subset \scr P_{\beta,\aa}$ for $k\ge 1$ and
\begin{align} \label{alw}\ff 1 {2(\alpha(1)+1)} \W_{\beta,\alpha}(\mu,\nu)\le \W_{k,var}(\mu,\nu):=\sup_{|f|\le 1+|\cdot|^k} |\mu(f)-\nu(f)|, \ \ \mu,\nu\in\scr P_k.
\end{align}
To obtain the Harnack inequality, we make the following assumptions.
\begin{enumerate}
\item[\bf{(A1)}] $\sigma_t$ is invertible and $\|\sigma_t\|+\|\sigma_t^{-1}\|$ is bounded in $t\in[0,T]$.
\item[\bf{(A2)}] For any $t\in[0,T],\gamma\in\scr P_2$, $\nabla B_t(\cdot,\gamma)$ is continuous. Moreover, there exist some constant $K_B>0$, $\alpha\in\scr A$ and $\beta\in(\frac{2}{3},1]$ such that
\begin{equation*}\begin{split}
&|\nabla B_t(x,\gamma)|\leq K_B, \\
& |B_t(x,\gamma)-B_t(x,\bar{\gamma})|\leq K_B(\W_2(\gamma,\bar{\gamma})+\W_{\beta,\alpha}(\gamma,\bar{\gamma})),\\
&|B_t(0,\delta_0)|\leq K_B,\ \ t\in[0,T],x\in\mathbb{R}^{m+d},\gamma,\bar{\gamma}\in\scr P_2.
\end{split}\end{equation*}
\item[\bf{(A3)}] $MM^\ast$ is invertible.
\end{enumerate}
By {\bf(A2)} and \eqref{alw} for $k=1$, there exist constants $C_1,C_2>0$ such that
\begin{equation}\label{BBS}\begin{split}|B_t(x,\gamma)|&\leq C_1(1+|x|+\W_2(\gamma,\delta_0)+\W_{\beta,\alpha}(\gamma,\delta_0))\\
&\leq C_2(1+|x|+\|\gamma\|_2),\ \ t\in[0,T],x\in\mathbb{R}^{m+d},\gamma\in\scr P_2.
\end{split}\end{equation}
So, according to  Theorem \ref{GWP} below for $k=2$, under {\bf(A1)}-{\bf(A2)}, \eqref{E0} is well-posed in $\scr P_2$, and $P_t^*{\gamma}:=\L_{X_t^{\gamma}}$ for the solution $X_t^{\gamma}$ to \eqref{E0} with $\L_{X_0^\gamma}=\gamma\in\scr P_2$ satisfy
\begin{align}\label{ga3}\|P_t^\ast\gamma\|_{2}^2\leq C_1(1+\|\gamma\|_2^2),\ \ t\in[0,T]
\end{align}
for some constant $C_1>0$.
Define
$$P_tf(\gamma):= \E[f(X_t^{\gamma})]=\int_{\R^{m+d}} f\d\{P_t^*{\gamma}\}.$$
For simplicity, we denote $X_t^{x}=X_t^{\delta_x}$ and $P_tf(x)=P_tf(\delta_x)$ for $x\in\R^{m+d}$.
The next result characterizes the Harnack inequality for \eqref{E0}.
\begin{thm}\label{LHI} Assume {\bf (A1)}-{\bf (A3)}.
Then the following assertions hold.
\begin{enumerate}
\item[(1)]
There exists a constant $c>0$ such that for any positive
$f\in \B_b(\R^{m+d})$,
\beg{equation*}\beg{split}
P_t\log f(\tilde{\gamma})&\leq\log P_t f(\gamma)+\left(\e^{c  \big(1+\|\gg\|_2+\|\tt\gg\|_2\big)^2}+\frac{c}{t^{3}}\right)\W_2(\gamma,\tilde{\gamma})^2,\ \ t\in (0,T],  \gamma,\tilde{\gamma}\in \scr P_{2}.\end{split}\end{equation*}
Consequently, it holds
\begin{equation*}\begin{split}
\|P_t^\ast\gamma-P_t^\ast\tilde{\gamma}\|_{var}^2&\leq2\mathrm{Ent}(P_t^\ast \gamma|P_t^\ast \tilde{\gamma})\\
&\leq \left(2\e^{c  \big(1+\|\gg\|_2+\|\tt\gg\|_2\big)^2}+\frac{2c}{t^{3}}\right)\W_2(\gamma,\tilde{\gamma})^2,\ \ t\in (0,T],  \gamma,\tilde{\gamma}\in \scr P_{2}.
\end{split}\end{equation*}
 \item[(2)] There exists $c>0$ such that for any $p>1$, $t\in (0,T],  \gamma,\tilde{\gamma}\in \scr P_{2}$ and $f\in\scr B_b^{+}(\R^{m+d})$,
 \beg{equation*}\beg{split}
(P_tf(\tilde{\gamma}))^p&\leq P_t f^p(\gamma)\exp\left\{\frac{cp}{(p-1)}\e^{c  \big(1+\|\gg\|_2+\|\tt\gg\|_2\big)^2}\W_2(\gamma,\tilde{\gamma})^2\right\}\\
&\qquad\quad\times \inf_{\pi\in\C(\gamma,\tilde{\gamma})}\int_{\R^{m+d}\times \R^{m+d}}\exp\left\{\frac{cp}{(p-1)t^3}|x-y|^2\right\}\pi(\d x,\d y).\end{split}\end{equation*}
\item[(3)] If in particular $B$ is bounded, then (1) and (2) hold for some constant $c>0$ replacing $\e^{c  \big(1+\|\gg\|_2+\|\tt\gg\|_2\big)^2}$.
 \end{enumerate}
 \end{thm}
 \section{Proof of Theorem \ref{LHI}}
 \subsection{The Bismut formula for stochastic Hamiltonian system with measure-valued curve parameter}
 The Bismut formula in the first assertion of the following theorem has been established in \cite{GW,Wbook,WZ1,Z1}, where the Malliavin calculus or coupling by change of measure plays an important role. For reader's convenience, we will use the coupling by change of measure to provide the proof. The second assertion in Theorem \ref{T3.2} below is new, which is crucial in the proof of Lemma \ref{I2t} below.

 For any $\mu\in C([0,T];\scr P_2)$, consider  the SDE with parameter $\mu$:
\begin{align}\label{DDT}\begin{cases}
\d (X^{x,\mu}_t)^{(1)}=M(X^{x,\mu}_t)^{(2)}\d t, \\
\d (X^{x,\mu}_t)^{(2)}=B_t(X^{x,\mu}_t,\mu_t)\d t+\sigma_t\d W_t,\ \ X^{x,\mu}_0=x\in\R^{m+d}.
\end{cases}
\end{align}
Observe that \eqref{DDT} is indeed a time non-homogeneous classical SDE. By {\bf (A1)}-{\bf (A2)} and \eqref{kvs} for $k=1$, \eqref{DDT} is well-posed. In fact, by {\bf (A2)} and \eqref{kvs} for $k=1$, there exists a constant $C>0$ such that
\begin{equation}\begin{split}\label{LPL}&|B_t(x,\mu_t)-B_t(y,\mu_t)|\leq C|x-y|,\\
 &|B_t(0,\mu_t)|\leq C(1+\|\mu_t\|_2),\ \ x,y\in\R^{m+d}, t\in[0,T].
\end{split}\end{equation}
Let $P_t^\mu$ be the associated Markov semigroup, i.e.
$$P_t^\mu f(x):=\E[f(X_t^{x,\mu})],\ \ t\in [0,T], x\in \R^{m+d}, f\in \B_b(\R^{m+d}).$$
\begin{thm}\label{T3.2} Assume {\bf (A1)}-{\bf(A2)}.
Then the following assertions hold.
\begin{enumerate}
\item[(1)] Suppose that {\bf(A3)} holds. Then for any $f\in\scr B_b(\R^{m+d})$, it holds
\begin{align}\label{GRm}\nabla_{h}P_t^\mu f(x)=\E [f(X_t^{x,\mu})N_t(h)],\ \ x, h\in\R^{m+d}, t\in(0,T]
\end{align}
with
$$N_t(h)=\int_0^t\left\<\sigma^{-1}_s\left[\nabla B_s(X_s^{x,\mu},\mu_s)\left(h^{(1)}+\int_0^sM\gamma_u(h),\ \ \gamma_s(h)\right)-\gamma_s'(h)\right],\d W_s\right\>$$
and \begin{equation}\begin{split}\label{gah}\gamma_s(h)&:=  \left[\ff{(t-s)}{t}-\ff{3s(t-s)}{t^2}M^*(MM^*)^{-1}M \right]h^{(2)}\\
&\qquad\quad-\ff{6s(t-s)}{t^3}M^*(MM^*)^{-1} h^{(1)}, \ \ s\in[0,t]
\end{split}\end{equation}
satisfying
$$|\gamma_s(h)|\leq c(|h^{(2)}|+t^{-1}|h^{(1)}|),\ \ |\gamma'_s(h)|\leq c(t^{-1}|h^{(2)}|+t^{-2}|h^{(1)}|),\ \ s\in[0,t]$$
for some constant $c>0$.
\item[(2)] For any $f\in\scr B_b(\R^{m+d})$ with $f(x)$ only depending on $x^{(2)}$ and $v\in\R^m$,
\begin{align}\label{NA1}\nabla^{(1)}_vP_t^\mu f(x)=\E \left[f((X_t^{x,\mu})^{(2)})\int_0^t\<\sigma^{-1}_s\nabla ^{(1)}_v B_s(X_s^{x,\mu},\mu_s), \d W_s\>\right],\ \ t\in[0,T].
\end{align}
Consequently, for any $f\in\scr B_b(\R^{m+d})$ with $f(x)$ only depending on $x^{(2)}$,
\begin{align}\label{N12}|\nabla^{(1)}P_t f(x)|\leq C\left(\E |f((X_t^{x,\mu})^{(2)})|^2\right)^{\frac{1}{2}}t^{\frac{1}{2}},\ \ t\in[0,T]
\end{align}
for some constant $C>0$.
\end{enumerate}
\end{thm}
 \begin{proof} (1) Fix $t\in(0,T]$.
For any $\vv\in(0,1)$, $h\in\R^{m+d}$, let $(X_s^\vv)_{s\in[0,t]}$ solve the equation
\beq\label{EC10} \beg{cases} \d (X_s^\vv)^{(1)}= M(X_s^\vv)^{(2)}\d s,\\
\d (X_s^\vv)^{(2)}= B_s(X_s^{x,\mu},\mu_s)\d s +\si_s\d W_s+ \vv \gamma'_s(h)\d s,\ \ X_0^\vv= x+\vv h.\end{cases}
\end{equation}
Then it is easy to see that
\beq\label{EE} X^\vv_s=X^{x,\mu}_s+\left(\vv h^{(1)}+\vv\int_0^{s}  M\gamma_u(h)\d u,\ \ \vv\gamma_s(h)\right),\ \
s\in[0,t],\end{equation}
in particular, $X^\vv_t=X^{x,\mu}_t$ due to \eqref{gah}.
Let
$$
\Phi^\vv_s=\sigma_s^{-1}[B_s(X^\vv_s,\mu_s)-B_s(X^{x,\mu}_s,\mu_s)-\vv\gamma'_s(h)],\ \ s\in[0,t].
$$
Then {\bf(A1)}-{\bf(A2)} and \eqref{EE} imply
\begin{align}\label{DDW}
|\Phi^\vv_s|\leq c_0\left[\varepsilon|h^{(1)}|+\vv\|M\|\int_0^s|\gamma_u(h)|\d u+\vv|\gamma_s(h)|\right]+\vv|\gamma'_s(h)|,\ \ s\in[0,t]
\end{align}
for some constant $c_0>0$.
In view of
\begin{align}\label{gammh}|\gamma_s(h)|\leq c(|h^{(2)}|+t^{-1}|h^{(1)}|),\ \ |\gamma'_s(h)|\leq c(t^{-1}|h^{(2)}|+t^{-2}|h^{(1)}|),\ \ s\in[0,t]
\end{align}
for some constant $c>0$, it follows from Girsanov's theorem that
$$
\tilde{W}_s:=W_s-\int_{0}^{s}\Phi^\vv_u\d u
,\ \ s\in[0,t]$$
is a $d$-dimensional Brownian motion on $[0,t]$ under $\Q_t^\vv=R^\vv_t\P$, where
\begin{align*}
R^\vv_t=\exp\bigg[\int_0^t\<\Phi^\vv_u, \d
W_u\>-\frac{1}{2}\int_0^t |\Phi^\vv_u|^2\d u\bigg].
\end{align*}
Then (\ref{EC10}) reduces to
\begin{equation*}  \beg{cases} \d (X^\vv_s)^{(1)}= M(X^\vv_s)^{(2)}\d s,\\
\d (X^\vv_s)^{(2)}=B_s(X^\vv_s,\mu_s)\d s +\si_s\d \tilde{W}_s,\ \ X_0^\vv= x+\vv h,
\end{cases}
\end{equation*}
which yields that the law of $X_t^\vv$ under $\Q^\vv_t$ coincides with that of $X^{x+\vv h,\mu}_t$  under $\P$.
As a result, we get
\begin{align*}
P_t^\mu f(x+\vv h)&=\E^{\Q^\vv_t} f(X_t^\vv)=\E ^{\Q^\vv_t}f(X^{x,\mu}_t)=\E [R^\vv_tf(X^{x,\mu}_t)],\ \ f\in\scr B_b(\R^{m+d}).
\end{align*}
Due to \eqref{DDW} and \eqref{gammh}, it is not difficult to verify that $\{\frac{R^\vv_t-1}{\vv}\}_{\vv\in(0,1)}$ is uniformly integrable and hence applying the dominated convergence theorem, {\bf(A1)}-{\bf(A3)} and \eqref{EE},
we have
$$\lim_{\vv\to0}\E\left|\frac{R^\vv_t-1}{\vv}-N_t(h)\right|=0,$$
which derives \eqref{GRm}.
This combined with \eqref{gammh} completes the proof.

(2) For any $x\in\R^{m+d}, v\in\R^m$, let $\hat{X}_t^\vv=((X_t^{x,\mu})^{(1)}+\vv v,(X_t^{x,\mu})^{(2)})$. Then it is clear that
\begin{equation*}\beg{cases} \d (\hat{X}_s^\vv)^{(1)}= M(\hat{X}_s^\vv)^{(2)}\d s,\\
\d (\hat{X}_s^\vv)^{(2)}= B_s(X_s^{x,\mu},\mu_s)\d s +\si_s\d W_s,\ \ \hat{X}_0^\vv=(x^{(1)}+\vv v,x^{(2)}).\end{cases}
\end{equation*}
Rewrite it as
\beq\label{ECv} \beg{cases} \d (\hat{X}_s^\vv)^{(1)}= M(\hat{X}_s^\vv)^{(2)}\d s,\\
\d (\hat{X}_s^\vv)^{(2)}= B_s(\hat{X}_s^\vv,\mu_s)\d s +\si_s\d \hat{W}_s,\ \ \ \ \hat{X}_0^\vv=(x^{(1)}+\vv v,x^{(2)}),\end{cases}
\end{equation}
with
$$\d \hat{W}_s=\d W_s-\eta_s^\vv\d s,\ \ \eta_s^\vv=\sigma_s^{-1}[B_s(\hat{X}_s^\vv,\mu_s)-B_s(X_s^{x,\mu},\mu_s)],\ \ s\in[0,t].$$
Then {\bf(A1)}-{\bf(A2)} gives
\begin{align}\label{DDN}|\eta_s^\vv|\leq \sup_{s\in[0,T]}\|\sigma^{-1}_s\|K_B\vv|v|,\ \ s\in[0,t].
\end{align}
Let $$
\hat{R}^\vv_t=\exp\bigg[\int_0^t\< \eta_u^\vv, \d
W_u\>-\frac{1}{2}\int_0^t |\eta^\vv_u|^2\d u\bigg].
$$
Girsanov's theorem yields that $(\hat{W}_s)_{s\in[0,t]}$ is a $d$-dimensional Brownian motion under $\hat{\Q}^\vv_t=\hat{R}^\vv_t\P$ and so \eqref{ECv} implies that the law of $\hat{X}_t^\vv$ under $\hat{\Q}^\vv_t$ coincides with that of $X^{(x^{(1)}+\vv v, x^{(2)}),\mu}_t$  under $\P$, which together with $(\hat{X}_t^\vv)^{(2)}=(X^{x,\mu}_t)^{(2)}$ yields that for any $f\in\scr B_b(\R^{m+d})$ with $f(x)$ only depending on $x^{(2)}$,
\begin{align*}
P_t^\mu f(x^{(1)}+\vv v, x^{(2)})&=\E^{\hat{\Q}^\vv_t} f((\hat{X}_t^\vv)^{(2)})=\E ^{\hat{\Q}^\vv_t}f((X^{x,\mu}_t)^{(2)})=\E [\hat{R}^\vv_tf((X^{x,\mu}_t)^{(2)})].
\end{align*}
Similarly to (1), by \eqref{DDN}, one may verify that $\{\frac{\hat{R}^\vv_t-1}{\vv}\}_{\vv\in(0,1)}$ is uniformly integrable, which together with the dominated convergence theorem and {\bf(A1)}-{\bf(A2)} yields
$$\lim_{\vv\to0}\E\left|\frac{\hat{R}^\vv_t-1}{\vv}-\int_0^t\<\sigma^{-1}_s\nabla_v^{(1)} B_s(X_s^{x,\mu},\mu_s), \d W_s\>\right|=0,$$
 This implies \eqref{NA1}. Finally, \eqref{N12} follows from \eqref{NA1}, Cauchy-Schwarz's inequality and {\bf (A1)}-{\bf (A2)}.
 \end{proof}
\subsection{Proof of Theorem \ref{LHI}}
For any $\gg\in \scr P_2$, consider  the   decoupled SDE:
\begin{align}\label{dec}\begin{cases}
\d (X^{x,\gamma}_t)^{(1)}=M(X^{x,\gamma}_t)^{(2)}\d t, \\
\d (X^{x,\gamma}_t)^{(2)}=B_t(X^{x,\gamma}_t,P_t^\ast\gamma)\d t+\sigma_t\d W_t,\ \ X^{x,\gamma}_0=x\in\R^{m+d}.
\end{cases}
\end{align}
\eqref{ga3} together with {\bf (A2)} implies that \eqref{LPL} holds for $\mu_t=P_t^\ast \gamma$, so that SDE \eqref{dec} is well-posed and it is standard to derive that for any $p\geq 1$, there exists a constant $C>0$ such that
\begin{align}\label{sus}\E\sup_{t\in[0,T]}|X^{x,\gamma}_t|^p\leq C(1+|x|^p+\|\gamma\|_2^p).
\end{align}
Let $P_t^\gg$ be the associated Markov semigroup to \eqref{dec}, i.e.
$$P_t^\gg f(x):=\E[f(X_t^{x,\gg})],\ \ t\in [0,T], x\in \R^{m+d}, f\in \B_b(\R^{m+d}).$$
Then it holds
$$P_t f(\gamma):=\int_{\R^{m+d}} f(x)(P_t^*\gg)(\d x) =\int_{\R^{m+d}}P^{\gamma}_t f(x)\gamma(\d x),\ \ f\in\scr B_b(\R^{m+d}).$$
\begin{lem}\label{poa} Assume {\bf (A1)}-{\bf (A2)}.
Then for any $p\geq 1$, there exists a constant $c_p>0$ such that
\begin{equation}\begin{split}\label{soi}
&\E|(X^{x,\gamma}_t)^{(1)}-x^{(1)}-tM x^{(2)}|^p\\
&\leq c_p(1+|x|^p+\|\gg\|_2^p)t^{\frac{3p}{2}},\ \ t\in[0,T], x\in\R^{m+d}, \gg\in \scr P_2,
\end{split}\end{equation}
\begin{align}\label{sos}
\E\sup_{s\in[0,t]}|(X^{x,\gamma}_s)^{(2)}-x^{(2)}|^p\leq c_p(1+ |x|^p+\|\gg\|_2^p)t^{\frac{p}{2}},\ \ t\in[0,T], x\in\R^{m+d}, \gg\in \scr P_2.
\end{align}
If $B$ is bounded, then
\begin{align}\label{soi01}
\E|(X^{x,\gamma}_t)^{(1)}-x^{(1)}-tM x^{(2)}|^p\leq c_pt^{\frac{3p}{2}},\ \ t\in[0,T], x\in\R^{m+d}, \gg\in \scr P_2,
\end{align}
\begin{align}\label{sos01}
\E\sup_{s\in[0,t]}|(X^{x,\gamma}_s)^{(2)}-x^{(2)}|^p\leq c_pt^{\frac{p}{2}},\ \ t\in[0,T], x\in\R^{m+d}, \gg\in \scr P_2.
\end{align}
\end{lem}
\begin{proof} 
Observe that
\begin{equation*}
\begin{cases}
(X^{x,\gamma}_t)^{(1)}= x^{(1)}+\int_0^tM(X^{x,\gamma}_s)^{(2)}\d s, \\
(X^{x,\gamma}_t)^{(2)}=x^{(2)}+\int_0^tB_s(X^{x,\gamma}_s,P_s^\ast\gamma)\d s+\int_0^t\sigma_s\d W_s.
\end{cases}
\end{equation*}
We have
$$(X^{x,\gamma}_t)^{(1)}-x^{(1)}-tM x^{(2)}=\int_0^t M((X^{x,\gamma}_s)^{(2)}-x^{(2)})\d s.$$
So, it is sufficient to prove \eqref{sos} and \eqref{sos01}. When $B$ is bounded, it is easy to get \eqref{sos01} by the  BDG inequality. Furthermore,
it is not difficult to see from the BDG inequality, \eqref{BBS}, \eqref{ga3} and \eqref{sus} that \eqref{sos} holds.
\end{proof}
\begin{lem}\label{L1}
Assume {\bf(A2)}. Then there exists a constant $c>0$ such that
\begin{align}\label{WAr}
\W_2(P_t^\ast\gamma,P_t^\ast \tt\gamma)\leq c \W_2(\gamma,\tt\gamma)+c \int_0^t \W_{\beta,\alpha}(P_s^\ast\gamma,P_s^\ast \tt\gamma) \d s,\ \ t\in [0,T], \gg,\tt\gg\in \scr P_2.
\end{align}
\end{lem}

\begin{proof} Take $\F_0$-measurable random variables $X_0^\gg,X_0^{\tt\gg}$  such that
\beq\label{FGG} \L_{X_0^\gg}=\gg,\ \ \L_{X_0^{\tt\gg}}=\tt\gg,\ \ \W_2(\gg,\tt\gg)^2=\E|X_0^\gg-X_0^{\tt\gg}|^2.\end{equation}
By {\bf(A2)}, we find a constant $c_1>1$ such that
\beg{align*}\E\Big[\sup_{s\in [0,t]}  |X^{\tt\gamma}_s-X^{\gamma}_s|^2\Big]
&\leq c_1 \E|X^{\tt\gamma}_0-X^{\gamma}_0|^2 \\
  &+  c_1  \left(\int_0^t   \big\{\W_{\beta,\alpha}(P_s^\ast\gamma, P_s^*\tt\gg)+ \W_2(P_s^*\gg, P_s^*\tt\gg)\big\}\d s\right)^2\\
  &+c_1\E\int_0^t|X^{\tt\gamma}_s-X^{\gamma}_s|^2\d s,\ \ t\in [0,T].\end{align*}
So, it follows from the inequality $\W_2(P_s^*\gg, P_s^*\tt\gg)^2\leq \E |X^{\tt\gamma}_s-X^{\gamma}_s|^2$ and Gronwall's inequality that \eqref{WAr} holds.
\end{proof}

The following H\"older inequality for concave functions comes from \cite[Lemma 2.4]{HW22a}.

\beg{lem}\label{LN} Let $\aa: [0,\infty)\to [0,\infty)$ be concave. Then for any non-negative random variables $\xi$ and $\eta$,
\begin{align}\label{Alp}\E[\aa(\xi)\eta]\le \|\eta\|_{L^p(\P)} \aa\Big(\|\xi\|_{L^{\ff{p}{p-1}} (\P)}\Big),\ \ p\ge 1.\end{align}
\end{lem}
The following lemma is crucial in the proof of the desired Harnack inequality.
\begin{lem}\label{DLP}
Assume {\bf (A1)}-{\bf(A3)}. Then there exists a constant $c>0$ such that
 \begin{equation}\begin{split}\label{bao}
  \W_{\beta,\alpha}(P_t^\ast\gg,P_t^\ast\tt\gg)&\le     c  \W_2(\tt\gg,\gg)(1 +\|\gg\|_2+\|\tt\gg\|_2)\\
  &\quad\times\bigg\{\ff{ \aa(t^{\ff 1 2}) }{\ss t}+ t^{\frac{3(\beta-1)}{2}}+  \e^{c  \big(1+\|\gg\|_2+\|\tt\gg\|_2\big)^2}\bigg\},\ \
     t\in (0,T],  \gg,\tt\gg\in \scr P_2.
     \end{split}\end{equation}
 Consequently,  there exists a constant $\tilde{c}>0$ such that for any $\gg,\tt\gg\in \scr P_2,$
\begin{equation}\begin{split}\label{TTH}\sup_{t\in [0,T]}  &\W_2(P_t^\ast\gg,P_t^\ast\tt\gg)
\le \e^{\tilde{c}  \big(1+\|\gg\|_2+\|\tt\gg\|_2\big)^2} \W_2(\gg,\tt\gg).
\end{split}\end{equation}
If $B$ is bounded, we have
\begin{align}\label{bal}\W_{\beta,\alpha}(P_t^\ast\gg,P_t^\ast\tt\gg)\leq c\W_2(\tt\gg,\gg)\Big\{\ff{ \aa(t^{\ff 1 2}) }{\ss t}+t^{\frac{3(\beta-1)}{2}}\Big\}
\end{align}
and
\begin{align}\label{wts}\sup_{t\in [0,T]}  \W_2(P_t^\ast\gg,P_t^\ast\tt\gg)\leq c\W_2(\tt\gg,\gg).\end{align}
\end{lem}
Let $X_0^\gg$ and $X_0^{\tt\gg}$ be in \eqref{FGG}.
For any $\vv\in [0,2]$, let
$$ X_0^{\gg^\vv}:=X_0^\gg+\vv(X_0^{\tt\gg}-X_0^\gg), \ \ \ \gg^\vv:=\L_{X_0^{\gg^\vv}}.$$
By the definition of $\W_{\beta,\alpha}$, for any $\vv,r\in[0,1]$ and $t\in[0,T]$, we have
\begin{equation}\begin{split}\label{I12ts}
\W_{\beta,\alpha}(P_t^\ast\gamma^{\varepsilon+r},P_t^\ast\gamma^{\varepsilon})^2&=\sup_{f\in\scr B_b(\R^{m+d}), [f]_{\beta,\alpha}\leq 1}|P_tf (\gamma^{\varepsilon+r})-P_tf(\gamma^\vv)|^2\\
&\leq 2\sup_{f\in\scr B_b(\R^{m+d}),[f]_{\beta,\alpha}\leq 1} |\gamma^{\varepsilon+r} (P_t^{\gamma^{\varepsilon+r}}f-P^{\gamma^\vv}_tf)|^2\\
 &\quad+2\sup_{f\in\scr B_b(\R^{m+d}),[f]_{\beta,\alpha}\leq 1}|\gamma^{\varepsilon+r}(P_t^{\gamma^\vv}f)-\gamma^\vv (P^{\gamma^\vv}_tf)|^2\\
&=:I_1(t)+I_2(t).
\end{split}\end{equation}
Therefore, to prove Lemma \ref{DLP}, it is sufficient to derive the estimates for $I_1(t)$ and $I_2(t)$, which will be provided in the following two lemmas.
\begin{lem}\label{I1t}Assume {\bf (A1)}-{\bf(A2)}. Then there exists a constant $c>0$ such that \beq\label{I1} \begin{split}I_1(t) &\le  c(1+ \|\gg\|_2+\|\tt\gg\|_2)^2\psi(\vv,r)\\
 &\quad\times\bigg(r^2 \W_2(\gg,\tt\gg)^2+ \int_0^t \W_{\beta,\alpha}(P_s^\ast\gamma^\vv,P_s^\ast\gamma^{\vv+r})^2\d s\bigg),\ \ t\in [0,T].\end{split} \end{equation}
When $B$ is bounded, we conclude that
\beq\label{I12} \begin{split}&I_1(t) \le  c\psi(\vv,r) \bigg(r^2 \W_2(\gg,\tt\gg)^2+ \int_0^t \W_{\beta,\alpha}(P_s^\ast\gamma^\vv,P_s^\ast\gamma^{\vv+r})^2\d s\bigg),\ \ t\in [0,T],\end{split} \end{equation}
\end{lem}
where
\beq\label{bar0} \psi(\vv,r):=  \e^{cr^2\W_2(\gamma,\tilde{\gamma})^2+c\int_0^T\W_{\beta,\alpha}(P_s^\ast\gamma^\vv,P_s^\ast\gamma^{\vv+r})^2\d s}.\end{equation}
\begin{proof}
Firstly, by the definition of $\gamma^\vv$, we have
\beq\label{GG2}  \|\gamma^\vv\|_2^2\le 8\|\gg\|_2^2+8\|\tt\gg\|_2^2,\ \ \vv\in [0,2],\end{equation}
and
\beq\label{TT}  \W_2(\gg^{\vv},\gg^{\vv+r})^2\leq  r^2 \W_2(\gg,\tt\gg)^2,\ \ \vv,r\in [0,1].\end{equation}
For any $\vv\in[0,2]$,
consider the SDE
\begin{align}\label{Evv}\begin{cases}
\d (X^{x,\gg^\vv}_t)^{(1)}=M(X^{x,\gg^\vv}_t)^{(2)}\d t, \\
\d (X^{x,\gg^\vv}_t)^{(2)}=B_t(X^{x,\gg^\vv}_t,P_t^\ast\gg^\vv)\d t+\sigma_t\d W_t,\ \ X^{x,\gg^\vv}_0=x\in\R^{m+d},t\in[0,T].
\end{cases}
\end{align}
For any $r,\vv\in [0,1]$, define
\beg{align*}  \eta_t^{\vv,r}=\sigma_t^{-1} [B_t(X^{x,\gamma^\vv}_t, P_t^*\gg^{\vv+r})-B_t(X^{x,\gamma^\vv}_t, P_t^*\gg^\vv)],\ \ t\in [0,T].
\end{align*}
By {\bf (A1)}-{\bf(A2)}, \eqref{alw} for $k=1$, \eqref{ga3} and \eqref{GG2}, there exist constants $c_1, c_2>0$ such that
\begin{equation}\begin{split}\label{etk}\sup_{t\in[0,T] }|\eta^{\varepsilon,r}_t|&\leq c_1\sup_{t\in[0,T]}\big\{\W_{\beta,\alpha}(P_t^\ast\gamma^\varepsilon, P_t^\ast\gamma^{\vv+r})+\W_2(P_t^\ast\gamma^\varepsilon, P_t^\ast\gamma^{\vv+r})\big\}\\
&\leq c_2(1+\|\gamma\|_2+\|\tilde{\gamma}\|_2),\ \ r,\vv\in [0,1].
\end{split}\end{equation}
It follows from Girsanov's theorem that
   $$W_t^{\vv,r}=  W_t- \int_0^t \eta_s^{\vv,r}\d s,\ \ \ t\in [0,T]$$
 is a $d$-dimensional Brownian motion under  the probability  $\Q^{\vv,r}:=R^{\varepsilon,r}_T \P$ with
$$R^{\varepsilon,r}_t:=\exp\left\{\int_0^t\<\eta^{\varepsilon,r}_s,\d W_s\>-\frac{1}{2}\int_0^t|\eta^{\varepsilon,r}_s|^2\d s\right\}, \ \ t\in[0,T].$$
Therefore, \eqref{Evv} can be reformulated as
\begin{align*}\begin{cases}
\d (X^{x,\gg^\vv}_t)^{(1)}=M(X^{x,\gg^\vv}_t)^{(2)}\d t, \\
\d (X^{x,\gg^\vv}_t)^{(2)}=B_t(X^{x,\gg^{\vv}}_t,P_t^\ast\gg^{\vv+r})\d t+\sigma_t\d W_t^{\vv,r},\ \ X^{x,\gg^\vv}_0=x\in\R^{m+d},t\in[0,T].
\end{cases}
\end{align*}
So, for any $f\in \scr B_b(\R^{m+d})$, it holds
\begin{align*} &P_t^{\gamma^{\varepsilon+r}}f(x)-P^{\gamma^\vv}_tf(x)\\
&=\mathbb{E} \left[f(X^{x,\gamma^\vv}_t)(R_t^{\varepsilon,r}-1)\right]\\
&=\E \left[[f(X^{x,\gamma^\vv}_t)-f(x^{(1)}+tMx^{(2)}, x^{(2)})](R_t^{\varepsilon,r}-1)\right],\ \  \varepsilon,r\in(0,1], t\in[0,T], x\in\R^{m+d}.
\end{align*}
Moreover, by \eqref{etk}, \eqref{WAr} and \eqref{TT}, we obtain
\beq\label{RTo}\beg{split}
 &\E|R_t^{\varepsilon,r}-1|^2= \E\big[(R_t^{\varepsilon,r})^2-1\big]
 \leq \mathrm{esssup}_{\Omega}(\e^{\int_0^t|\eta_s^{\varepsilon,r}|^2\d s}-1)\\
 &\leq \mathrm{esssup}_{\Omega}\left(\e^{\int_0^t|\eta_s^{\varepsilon,r}|^2\d s}\int_0^t|\eta_s^{\varepsilon,r}|^2\d s\right)\\
 &\leq \psi(\vv,r)\int_0^t\big\{\W_{\beta,\alpha}(P_s^\ast\gamma^\vv,P_s^\ast\gamma^{\vv+r})^2+\W_2(P_s^\ast\gamma^\vv,P_s^\ast\gamma^{\vv+r})^2\big\}\d s\\
 &\leq c_3 \psi(\vv,r)\left(r^2\W_2(\gamma,\tilde{\gamma})^2+\int_0^t\W_{\beta,\alpha}(P_s^\ast\gamma^\vv,P_s^\ast\gamma^{\vv+r})^2\d s\right),\ \ t\in[0,T]
\end{split}\end{equation}
for some constant $c_3>0$.
By    \eqref{etk},  we have
\beq\label{bar} \bar\psi :=\sup_{\vv,r\in [0,1]}\psi(\vv,\gg)<\infty. \end{equation}
Combining \eqref{kvs} for $k=1$, the Cauchy-Schwarz inequality and \eqref{RTo},  we arrive at
\beg{align*}
&\sup_{f\in \scr B_b(\R^{m+d}),[f]_{\beta,\alpha}\leq 1} |\gamma^{\varepsilon+r} (P_t^{\gamma^{\varepsilon+r}}f-P^{\gamma^\vv}_tf)|^2\\
&\leq\bigg(\int_{\R^{m+d}} \sup_{f\in \scr B_b(\R^{m+d}),[f]_{\beta,\alpha}\le 1}  \Big|\E \left[(f(X^{x,\gamma^\vv}_t)-f(x^{(1)}+tMx^{(2)}, x^{(2)}))(R_t^{\varepsilon,r}-1)\right]\Big| \gg^{\vv+r}(\d x)\bigg)^2 \\
&\le   \bigg(\int_{\R^{m+d}}\{2(2(\alpha(1)+1))^2\E(1+|X^{x,\gamma^\vv}_t-(x^{(1)}+tMx^{(2)}, x^{(2)})|^2)\}^{\frac{1}{2}}  \gg^{\vv+r}(\d x)\bigg)^2\\
&\qquad\quad\times \sup_x\E[|R_t^{\varepsilon,r}-1|^2] \\
&\le   \bigg(\int_{\R^{m+d}}\{2(2(\alpha(1)+1))^2\E(1+|X^{x,\gamma^\vv}_t-(x^{(1)}+tMx^{(2)}, x^{(2)})|^2)\}^{\frac{1}{2}}  \gg^{\vv+r}(\d x)\bigg)^2\\
&\qquad\quad\times   c_3\psi(\vv,r)\left(r^2\W_2(\gamma,\tilde{\gamma})^2+\int_0^t\W_{\beta,\alpha}(P_s^\ast\gamma^\vv,P_s^\ast\gamma^{\vv+r})^2\d s\right),\ \ t\in [0,T].
\end{align*}
This together with  \eqref{soi}, \eqref{sos} and \eqref{GG2}, we derive from Jensen's inequality that
\eqref{I1} holds for some constant $c>0$.
When $B$ is bounded, applying \eqref{soi01} and \eqref{sos01} replacing \eqref{soi} and \eqref{sos} respectively, we derive \eqref{I12}.
\end{proof}

\begin{lem}\label{I2t} Assume {\bf (A1)}-{\bf(A3)}. Then there exists a constant $c>0$ such that \beg{equation} \label{I23}\begin{split} I_2(t)
&\le cr^2  \W_2(\gamma,\tilde{\gamma})^2 (1+ \|\gg\|_2+ \|\tilde{\gg}\|_2)^{2}\bigg(t^{3(\beta-1)}+\frac{\aa\big(t^{\ff 1 2}\big)^2}{t}\bigg),\ \ t\in(0,T].\end{split}\end{equation}
When $B$ is bounded, it holds
\beg{equation} \label{I2s}\begin{split} I_2(t)
&\le cr^2  \W_2(\gamma,\tilde{\gamma})^2 \bigg(t^{3(\beta-1)}+\frac{\aa\big(t^{\ff 1 2}\big)^2}{t}\bigg)\ \ t\in(0,T].\end{split}\end{equation}
\end{lem}
\begin{proof}
 Note that by Theorem \ref{T3.2}, for any $p>1$,  there exists some constant $c_1>0$ such that for $h\in\R^{m+d}$,
\begin{align}\label{12h}\left(\E(N_t(h^{(1)},0)^p)\right)^{\frac{1}{p}}\leq c_1|h^{(1)}|t^{-\frac{3}{2}},\ \ \left(\E(N_t(0,h^{(2)})^p)\right)^{\frac{1}{p}}\leq c_1|h^{(2)}|t^{-\frac{1}{2}}, \ \ t\in(0,T].
\end{align}
For any $x_0\in\R^{m+d}$ and any $f\in \scr B_b(\R^{m+d})$, let
\begin{align*}&f^{(1)}_{x_0}(x)=f(x)-f(x_0^{(1)}+Mtx_0^{(2)}, x^{(2)}),\\
&\tilde{f}^{(1)}_{x_0}(x)=f(x_0^{(1)}+Mtx_0^{(2)}, x^{(2)})-f(x_0^{(1)}+Mtx_0^{(2)}, x_0^{(2)}),\\
 &f^{(2)}_{x_0}(x)=f(x)-f(x_0^{(1)}+Mtx_0^{(2)}, x_0^{(2)}),\ \  x\in\R^{m+d}, t\in[0,T].
\end{align*}
By Theorem \ref{T3.2}(1), \eqref{soi}, the first inequality in \eqref{12h},  \eqref{GG2} and H\"{o}lder's inequality, we have
\begin{align}\label{TTL}&\sup_{f\in \scr B_b(\R^{m+d}),[f]_{\beta,\aa}\le 1}|\nabla^{(1)}P_t^{\gamma^\vv}f^{(1)}_{x_0}|(x_0)\leq c_2(1+ |x_0|+ \|\gg\|_2+ \|\tilde{\gg}\|_2)^\beta t^{\frac{3(\beta-1)}{2}},\ \ t\in(0,T]
\end{align}
for some constant $c_2>0$.
By \eqref{N12}, \eqref{sos}, \eqref{GG2} and \eqref{DDY}, we can find constants $c_3, c_4>0$ such that
\begin{equation}\begin{split}\label{TTY}\sup_{f\in \scr B_b(\R^{m+d}),[f]_{\beta,\aa}\le 1}|\nabla^{(1)}P_t^{\gamma^\vv}\tilde{f}^{(1)}_{x_0}|(x_0)&\leq c_3\{1+\E|(X_t^{x_0,\gamma^\vv})^{(2)}-x_0^{(2)}|^2\}^{\frac{1}{2}}t^{\frac{1}{2}}\\
&\leq c_4(1+ |x_0|+ \|\gg\|_2+ \|\tilde{\gg}\|_2),\ \ t\in[0,T].
\end{split}\end{equation}
Moreover, by Theorem \ref{T3.2}(1), the second inequality in \eqref{12h}, \eqref{soi}, \eqref{sos}, \eqref{GG2}, \eqref{Alp} and \eqref{DDY}, we conclude
\begin{align}\begin{split}\label{TTW}&\sup_{f\in \scr B_b(\R^{m+d}),[f]_{\beta,\aa}
\le 1}|\nabla^{(2)}P_t^{\gamma^\vv}f_{x_0}^{(2)}|(x_0)\\
&\qquad\quad\leq c_5(1+ |x_0|+ \|\gg\|_2+ \|\tilde{\gg}\|_2)\left( t^{\frac{3\beta-1}{2}}+\frac{\alpha(t^{\frac{1}{2}})}{t^{\frac{1}{2}}}\right),\ \ t\in(0,T]
\end{split}\end{align}
for some constant $c_5>0$.
Since $$\nabla^{(1)}P^{\gamma^\vv}_t f=\nabla^{(1)}P^{\gamma^\vv}_t f^{(1)}_{x_0}+\nabla^{(1)}P^{\gamma^\vv}_t \tilde{f}^{(1)}_{x_0},\ \ \nabla^{(2)}P^{\gamma^\vv}_t f=\nabla^{(2)}P^{\gamma^\vv}_t f^{(2)}_{x_0},\ \ f\in \scr B_b(\R^{m+d}),$$
we derive from \eqref{TTL} and \eqref{TTY} that
\begin{equation}\begin{split}\label{TTU}&\sup_{f\in \scr B_b(\R^{m+d}),[f]_{\beta,\aa}\le 1}|\nabla^{(1)}P^{\gamma^\vv}_tf|(x_0)\leq c_6(1+ |x_0|+ \|\gg\|_2+ \|\tilde{\gg}\|_2) t^{\frac{3(\beta-1)}{2}},\ \ t\in(0,T]
\end{split}\end{equation}
and from \eqref{TTW} that
\begin{equation}\begin{split}\label{TTN}\sup_{f\in \scr B_b(\R^{m+d}),[f]_{\beta,\aa}\le 1}|\nabla^{(2)}P^{\gamma^\vv}_tf|(x_0)\leq c_5(1+ |x_0|+ \|\gg\|_2+ \|\tilde{\gg}\|_2)\left( t^{\frac{3\beta-1}{2}}+\frac{\alpha(t^{\frac{1}{2}})}{t^{\frac{1}{2}}}\right)
\end{split}\end{equation}
for some constant $c_6>0$. Observe that
\begin{equation*}\beg{split}
&I_{2}(t)= 2\sup_{f\in \scr B_b(\R^{m+d}),[f]_{\beta,\aa}\le 1}\bigg|\E\int_0^r\ff{\d}{\d \theta} P_t^{\gg^\vv} f(X_0^{\gamma^{\vv+\theta}})\d \theta\bigg|^2\\
&\quad=2 \sup_{f\in \scr B_b(\R^{m+d}),[f]_{\beta,\aa}\le 1} \bigg|\E\int_0^r \Big\{\nn_{X_0^{\tt\gg}-X_0^\gg} P_t^{\gg^\vv} f(X_0^{\gg^{\vv+\theta}}) \Big\}\d \theta\bigg|^2.\end{split}
\end{equation*}
Combining this with \eqref{TTU}, \eqref{TTN}, \eqref{GG2} and \eqref{FGG},
 we find constants $c_7,c_8>0$ such that
\beg{equation*} \begin{split} &I_2(t) \le 2 \sup_{f\in \scr B_b(\R^{m+d}),[f]_{\beta,\aa}\le 1}\bigg(\E\bigg[|X_0^\gg-X_0^{\tt\gg}|\int_0^r |\nn P_t^{\gg^\vv} f(X_0^{\gg^{\vv+\theta}})|\d \theta \bigg]\bigg)^2\\
&\le  c_7 \bigg\{\int_0^ r \bigg[\|X_0^\gg-X_0^{\tt\gg}\|_{L^2(\P)}(1+\|X_0^{\gg^{\vv+\theta}}\|_{L^2(\P)}+ \|\gg\|_2+\|\tilde{\gg}\|_2)\bigg(t^{\frac{3\beta-1}{2}}+\frac{\alpha(t^{\frac{1}{2}})}{t^{\frac{1}{2}}}\bigg)\bigg] \d \theta\bigg\}^2\\
&\le c_8r^2  \W_2(\gamma,\tilde{\gamma})^2(1+ \|\gg\|_2+ \|\tilde{\gg}\|_2)^{2}\bigg(t^{3(\beta-1)}+\frac{\alpha(t^{\frac{1}{2}})^2}{t}\bigg).\end{split}\end{equation*}
When $B$ is bounded, repeating the above procedure and using \eqref{soi01} and \eqref{sos01} replacing \eqref{soi} and \eqref{sos} respectively, we can obtain \eqref{I2s}.
\end{proof}
Now, we are in the position to prove Lemma \ref{DLP}.
\begin{proof} [Proof of Lemma \ref{DLP}]
Firstly, by $\alpha\in\scr A$ and  $\beta>\frac{2}{3}$, we conclude that
\beq\label{ASS} \beg{split} & \int_0^T \ff{  \aa( t^{\ff 1 2})^2}t \d t =   2   \int_0^{ T^{\ff 1 2} } \ff{\aa(s)^2}s\d s<\infty,\ \ \int_0^Tt^{3(\beta-1)}\d t<\infty.\end{split}\end{equation}
\eqref{I23} together with \eqref{I1} and \eqref{I12ts} yields
\begin{equation}\begin{split}\label{NNL}
&\W_{\beta,\alpha}(P_t^\ast\gamma^\vv,P_t^\ast\gamma^{\vv+r})^2\\
&\leq c(1+ \|\gg\|_2+\|\tt\gg\|_2)^2\psi(\vv,r) \int_0^t \W_{\beta,\alpha}(P_s^\ast\gamma^\vv,P_s^\ast\gamma^{\vv+r})^2\d s\\
&+cr^2  \W_2(\gamma,\tilde{\gamma})^2 (1+ \|\gg\|_2+ \|\tilde{\gg}\|_2)^{2}\bigg(\psi(\vv,r)+t^{3(\beta-1)}+\frac{\aa\big(t^{\ff 1 2}\big)^2}{t}\bigg),\ \ t\in(0,T].
\end{split}\end{equation}
Let
$$\GG_t(\vv,r):=\int_0^t\W_{\beta,\alpha}(P_s^\ast\gamma^\vv,P_s^\ast\gamma^{\vv+r})^2\d s.$$
So, it follows from \eqref{NNL} that
 \beq\label{LI3} \beg{split} &\GG_t(\vv,r)\le c  r^2\W_2(\gg,\tt\gg)^2H(\vv,r)
+  cF(\vv,r) \int_0^t \GG_s(\vv,r)\d s,\ \ t\in [0,T],\\
&H(\vv,r):=\big[1+\|\gg\|_2+\|\tt\gg\|_2\big]^2\int_0^T\left[\psi(\vv,r)+t^{3(\beta-1)}+\frac{\aa\big(t^{\ff 1 2}\big)^2}{t}\right]\d t,\\
 &F(\vv,r):=\big[1+\|\gg\|_2+\|\tt\gg\|_2\big]^2\psi(\vv,r),\ \ \   \vv,r\in [0,1].\end{split}\end{equation}
By Gronwall's inequality and \eqref{LI3},  for any $   \vv,r\in [0,1]$ we have
\beg{equation}\begin{split} \label{bes}&\GG_t(\vv,r) \le  c  r^2\W_2(\gg,\tt\gg)^2  \e^{cF(\vv,r)t}H(\vv,r),\ \ t\in [0,T].
 \end{split} \end{equation}
Substituting this into \eqref{NNL}, we get
 \begin{equation}\begin{split}\label{FFU}
\W_{\beta,\alpha}(P_t^\ast\gamma^\vv,P_t^\ast\gamma^{\vv+r})^2&\leq c(1+ \|\gg\|_2+\|\tt\gg\|_2)^2r^2  \W_2(\gamma,\tilde{\gamma})^2\\
&\times\bigg [c\psi(\vv,r)  \e^{cF(\vv,r)t}H(\vv,r)+\psi(\vv,r)+t^{3(\beta-1)}+\frac{\aa\big(t^{\ff 1 2}\big)^2}{t}\bigg].
\end{split}\end{equation}
Note that \eqref{bar}, \eqref{bar0}, \eqref{LI3} and \eqref{bes} imply that $\psi(\vv,r)$ is bounded in $(\vv,r)\in [0,1]^2$ with $\psi(\vv,r)\to 1$ as $r\to 0$, so that by \eqref{FFU} and the dominated convergence theorem
 we find a constant $C >1$ such that
\beq\label{AB} \beg{split}& \limsup_{r\downarrow 0}\ff{ \W_{\beta,\alpha}(P_t^*\gg^\vv, P_t^*\gg^{\vv+r})}r\\
 &  \le  C \W_2(\tt\gg,\gg)(1 +\|\gg\|_2+\|\tt\gg\|_2)\bigg\{\ff{ \aa(t^{\ff 1 2}) }{\ss t}+ t^{\frac{3(\beta-1)}{2}}+  \e^{C  \big(1+\|\gg\|_2+\|\tt\gg\|_2\big)^2}\bigg\}.
\end{split} \end{equation}
where we have used the fact that for some constant $C>1$,
\beg{align*}
&(1+\|\gg\|_2+\|\tt\gg\|_2)^2\e^{cT  (1+\|\gg\|_2+\|\tt\gg\|_2)^2}\le  \e^{C  (1+\|\gg\|_2+\|\tt\gg\|_2)^2}.
\end{align*}
 By the triangle inequality,
$$\big| \W_{\beta,\alpha}(P_t^*\gg, P_t^*\gg^\vv)- \W_{\beta,\alpha}(P_t^*\gg, P_t^*\gg^{\vv+r})\big|\le \W_{\beta,\alpha}(P_t^*\gg^\vv, P_t^*\gg^{\vv+r}),  \ \    \vv,r\in [0,1],$$
so that \eqref{AB}   implies that
  $ \W_{\beta,\alpha}(P_t^*\gg, P_t^*\gg^\vv)$ is Lipschitz continuous (hence a.e. differentiable)  in $\vv\in [0,1]$ for any $t\in (0,T]$,  and
 \beg{align*} & \Big|\ff{\d}{\d \vv} \W_{\beta,\alpha}(P_t^*\gg, P_t^*\gg^{\vv})\Big|\le   \limsup_{r\downarrow 0}\ff{ \W_{\beta,\alpha}(P_t^*\gg^\vv, P_t^*\gg^{\vv+r}) }r\\
 &\le  C \W_2(\tt\gg,\gg)(1 +\|\gg\|_2+\|\tt\gg\|_2)\bigg\{\ff{ \aa(t^{\ff 1 2}) }{\ss t}+ t^{\frac{3(\beta-1)}{2}}+  \e^{C  \big(1+\|\gg\|_2+\|\tt\gg\|_2\big)^2}\bigg\},\ \ \vv\in [0,1].\end{align*}
Noting that $\gg^1=\tt\gg$, this implies the desired estimate \eqref{bao}, which combined with \eqref{WAr} yields \eqref{TTH}.

 When $B$ is bounded, we derive from \eqref{I12}, \eqref{I2s} and \eqref{I23} that
\begin{equation}\begin{split}\label{NNW}
&\W_{\beta,\alpha}(P_t^\ast\gamma^\vv,P_t^\ast\gamma^{\vv+r})^2\\
&\leq c\psi(\vv,r) \int_0^t \W_{\beta,\alpha}(P_s^\ast\gamma^\vv,P_s^\ast\gamma^{\vv+r})^2\d s\\
&+cr^2  \W_2(\gamma,\tilde{\gamma})^2 \bigg(\psi(\vv,r)+t^{3(\beta-1)}+\frac{\aa\big(t^{\ff 1 2}\big)^2}{t}\bigg),\ \ t\in(0,T].
\end{split}\end{equation}
 Then by \eqref{NNW}, we have
 \beq\label{LIs} \beg{split}
 &\GG_t(\vv,r)\le c  r^2\W_2(\gg,\tt\gg)^2H(\vv,r)
+  cF(\vv,r) \int_0^t \GG_s(\vv,r)\d s,\ \ t\in [0,T],\\
&H(\vv,r):=\int_0^T\left[\psi(\vv,r)+ \frac{\aa\big(t^{\ff 1 2}\big)^2}{t}+t^{3(\beta-1)}\right]\d t,\\
 &F(\vv,r):=\psi(\vv,r),\ \ \   \vv,r\in [0,1].\end{split}\end{equation}
 Repeating the same procedure by replacing \eqref{LI3} with \eqref{LIs}, we derive \eqref{bal} and \eqref{wts}.
\end{proof}
Finally, we intend to prove Theorem \ref{LHI}.
\beg{proof}[Proof of Theorem \ref{LHI}]
Consider
\begin{align*}\begin{cases}
\d (X^{x,\gamma}_s)^{(1)}=M(X^{x,\gamma}_s)^{(2)}\d s, \\
\d (X^{x,\gamma}_s)^{(2)}=B_s(X^{x,\gamma}_s,P_s^\ast\gamma)\d s+\sigma_s\d W_s,\ \ X^{x,\gamma}_0=x.
\end{cases}
\end{align*}
Recall that $\gamma_s(h)$ is defined in \eqref{gah}.
Let $\tilde{X}_s$ solve
\begin{align}\label{Has}\begin{cases}
\d \tilde{X}_s^{(1)}=M\tilde{X}_s^{(2)}\d s, \\
\d \tilde{X}_s^{(2)}=B_s(X^{x,\gamma}_s,P_s^\ast\gamma)\d s+\sigma_s\d W_s+\gamma'_s(y-x)\d s,\ \ \tilde{X}_0=y.
\end{cases}
\end{align}
Then it holds
\begin{equation*}\tilde{X}_s=X^{x,\gamma}_s+\left( y^{(1)}-x^{(1)}+\int_0^{s}  M\gamma_u(y-x)\d u,\ \ \gamma_s(y-x)\right),\ \
s\in[0,t],\end{equation*}
in particular, $\tilde{X}_t=X^{x,\gamma}_t$ due to \eqref{gah}.
Let
\beg{align*}
&\eta_s^{\gg,\tt\gg}:=\sigma_s^{-1}[B_s(\tilde{X}_s, P_s^\ast\tilde{\gamma})-B_s(X^{x,\gamma}_s,P_s^\ast\gamma)-\gamma'_s(y-x)],\ \ s\in[0,t],\\
&R_t^{\gg,\tt\gg}:= \e^{\int_0^t\<\eta_s^{\gg,\tt\gg},\d W_s\> -\ff 1 2\int_0^t |\eta_s^{\gg,\tt\gg}|^2\d s}.\end{align*}
{\bf(A1)}-{\bf(A2)} imply
\begin{equation}\begin{split}\label{phs}
|\eta_s^{\gg,\tt\gg}|&\leq c_0\bigg(|\gamma_s(y-x)|+|y^{(1)}-x^{(1)}|+\|M\|\int_0^{s}  |\gamma_u(y-x)|\d u\\
&\qquad\quad+\W_2(P_s^\ast\gamma,P_s^\ast\tilde{\gamma})+\W_{\beta,\alpha}(P_s^\ast\gg,P_s^*\tt\gg)+|\gamma'_s(y-x)|\bigg),\ \ s\in[0,t]
\end{split}\end{equation}
for some constant $c_0>0$.
By \eqref{ASS} and  Lemma \ref{DLP}, there exists a constant $c_1>0$  such that
\beg{align*}&\int_0^T \big\{\W_{\beta,\alpha}(P_s^\ast\gg,P_s^*\tt\gg)^2+ \W_2(P_s^\ast\gg,P_s^*\tt\gg)^2\big\}\d s\le c_1   \e^{c_1  \big(1+\|\gg\|_2+\|\tt\gg\|_2\big)^2}\W_2(\gamma,\tilde{\gamma})^2.\end{align*}
This together with \eqref{phs} and \eqref{gammh} gives
\begin{align}\label{pas}
&\int_0^t|\eta_s^{\gg,\tt\gg}|^2\d s\leq\frac{c_2|x-y|^2}{t^3}+c_2   \e^{c_2  \big(1+\|\gg\|_2+\|\tt\gg\|_2\big)^2}\W_2(\gamma,\tilde{\gamma})^2
\end{align}
for some constant $c_2>0$.
As a result, Girsanov's theorem implies that
$$W_s^{\gamma,\tilde{\gamma}}=W_s-\int_0^s\eta_u^{\gg,\tt\gg}\d u,\ \ s\in[0,t]$$
is a $d$-dimensional Brownian motion under the probability measure $\mathbb{Q}_t^{\gamma,\tilde{\gamma}}=R_t^{\gg,\tt\gg}\P$.
So, \eqref{Has} can be rewritten as
\begin{align*}\begin{cases}
\d \tilde{X}_s^{(1)}=M\tilde{X}_s^{(2)}\d s, \\
\d \tilde{X}_s^{(2)}=B_s(\tilde{X}_s,P_s^\ast\tilde{\gamma})\d s+\sigma_s\d W_s^{\gamma,\tilde{\gamma}},\ \ \tilde{X}_0=y,
\end{cases}
\end{align*}
which derives
\begin{align*}
P_t^{\tilde{\gamma}} f(y)&=\E^{\Q_t^{\gg,\tt\gg}} f(\tilde{X}_t)=\E^{\Q_t^{\gg,\tt\gg}} f(X_t^{x,\gamma})=\E [R^{\gg,\tt\gg}_tf(X^{x,\gamma}_t)],\ \ f\in\scr B_b(\R^{m+d}).
\end{align*}
By Young's inequality, we have
\begin{align}\label{loi}\nonumber P_t^{\tilde{\gg}} \log f(y)&\le \log P_t^\gg f(x)+\E(R^{\gg,\tt\gg}_t\log R^{\gg,\tt\gg}_t)\\
&\leq \log P_t^\gg f(x)+\frac{1}{2}\E^{\Q_t^{\gg,\tt\gg}}\int_0^t|\eta_s^{\gg,\tt\gg}|^2\d s,\ \ f\in \B_b^+(\R^{m+d}),f>0.
\end{align}
H\"{o}lder's inequality yields that for any $p>1$,
\begin{align}\label{hap}
\nonumber(P_t^{\tilde{\gamma}}f(y))^p&\leq P_t^{\gamma} f^p(x)(\E (R^{\gamma,\tilde{\gamma}}_t)^{\frac{p}{p-1}})^{p-1}\\
& \leq P_t^\gamma f^p(x)\mathrm{esssup}_{\Omega}\exp\left\{\frac{p}{2(p-1)}\int_0^t |\eta_u^{\gamma,\tilde{\gamma}}|^2\d u\right\},\ \ f\in\scr B^+_b(\R^{m+d}).
\end{align}
Applying \eqref{pas}, taking expectation in \eqref{loi} and \eqref{hap} with respect to any $\pi\in\C(\gamma,\tilde{\gamma})$ and then taking infimum in $\pi\in\C(\gamma,\tilde{\gamma})$, the proof is completed by Jensen's inequality and \eqref{ETX}.
\end{proof}
\section{Well-posedness}
In this section, we consider a general version of \eqref{E0}.
Fix $T>0$ and let $k\geq 1$. Consider the distribution dependent SDEs on $\R^{m+d}$:
\beq\label{PDD}
\begin{cases}
\d X_t^{(1)}=b_t(X_t)\d t, \\
\d X_t^{(2)}=B_t(X_t,\L_{X_t})\d t+\sigma_t(X_t)\d W_t,
\end{cases}
\end{equation}
where $b:[0,T]\times\R^{m+d}\to \mathbb{R}^{m}$, $B:[0,T]\times\R^{m+d}\times \scr P\to \mathbb{R}^{d}$, $\sigma:[0,T]\times\R^{m+d}\to\mathbb{R}^{d}\otimes\mathbb{R}^{d}$ are measurable and $W_t$ is a $d$-dimensional Brownian motion on some complete filtration probability space $(\Omega, \scr F, (\scr F_t)_{t\geq 0},\P)$.
Let $C([0,T];\scr P_k)$ denote the continuous maps from $[0,T]$ to $(\scr P_k,\W_k)$.
\begin{defn}\label{defws} The SDE \eqref{PDD} is called well-posed for distributions in $\scr P_k$, if for any $\F_0$-measurable initial value
$X_0$ with $\L_{X_0}\in \scr P_k$ (respectively any initial distribution $\gamma\in\scr P_k$), it has a unique strong solution (respectively weak solution) such that $\L_{X_\cdot}\in C([0,T];\scr P_k)$.
\end{defn}
We make the following assumptions.
\begin{enumerate}
\item[\bf{(C1)}] For any $t\in[0,T],x\in\R^{m+d}$, $\sigma_t(x)$ is invertible and $\|\sigma^{-1}\|_\infty$ is finite.
\item[\bf{(C2)}] There exists $K>0$ such that
\begin{equation*}\begin{split}
&|b_t(x)-b_t(\bar{x})|+|\sigma_t(x)-\sigma_t(\bar{x})|\leq K|x-\bar{x}|,\\
&|B_t(x,\gamma)-B_t(\bar{x},\bar{\gamma})|\leq K(|x-\bar{x}|+\W_k(\gamma,\bar{\gamma})+\|\gamma-\bar{\gamma}\|_{k,var}),\\
&|b_t(0)|+|\sigma_t(0)|\leq K,\ \ |B_t(0,\gamma)|\leq K(1+\|\gamma\|_k),\ \ x, \bar{x}\in\mathbb{R}^{m+d},\gamma,\bar{\gamma}\in\scr P_k.
\end{split}\end{equation*}
\end{enumerate}
For any $\mu\in C([0,T],\scr P_k)$, consider
\beq\label{PDP}
\begin{cases}
\d X_t^{(1)}=b_t(X_t)\d t, \\
\d X_t^{(2)}=B_t(X_t,\mu_t)\d t+\sigma_t(X_t)\d W_t.
\end{cases}
\end{equation}
Under {\bf (C2)}, for any $\F_0$-measurable random variable $X_0$ with $\L_{X_0}\in\scr P_k$, let $X_t^{X_0,\mu}$ be the unique solution to \eqref{PDP} with initial value $X_0$.
It is standard to derive from {\bf (C2)} that
\begin{align}\label{cod}\E(\sup_{t\in[0,T]}|X_t^{X_0,\mu}|^n|\F_0)\leq c(n)(1+|X_0|^n),\ \ n\geq 1.
\end{align}
Define the mapping $\Phi^{X_0}:C([0,T],\scr P_k)\to C([0,T],\scr P_k)$ as
$$\Phi^{X_0}_t(\mu)=\L_{X_t^{X_0,\mu}},\ \ t\in[0,T].$$
The following theorem provides the well-posedness for \eqref{PDD} and the proof is similar to that in \cite[Theorem 3.2]{FYW2}.
\begin{thm}\label{GWP} Assume {\bf(C1)}-{\bf(C2)}. Then \eqref{PDD} is well-posed in $\scr P_k$. Moreover, there exists a constant $C>0$ such that
\begin{align}\label{pts}\|P_t^\ast\gamma\|_k^k\leq C(1+\|\gamma\|_k^k),\ \ t\in[0,T].
\end{align}
\end{thm}
\begin{proof} Since \eqref{pts} is standard by {\bf(C2)} and the BDG inequality, it is sufficient to prove that \eqref{PDD} is well-posed in $\scr P_k$. It follows from {\bf(C2)} that
\begin{align}\label{mnu00}
\nonumber|X^{X_0,\nu}_t-X^{X_0,\mu}_t|^k&\leq C_0\left(\int_0^t[\W_k(\mu_s,\nu_s)+\|\mu_s-\nu_s\|_{k,var}]\d s\right)^k+C_0\int_0^t|X^{X_0,\nu}_s-X^{X_0,\mu}_s|^k\d s\\
&+C_0\left|\int_0^t[\sigma_s(X^{X_0,\nu}_s) -\sigma_s(X_s^{X_0,\mu})]\d W_s\right|^k
\end{align}
for some constant $C_0>0$.
By {\bf(C2)} and the BDG inequality, there exist constants $C_1,C_2>0$ such that
\begin{equation}\begin{split}\label{dif}
&C_0\E\sup_{t\in[0,r]}\left|\int_0^t[\sigma_s(X^{X_0,\mu}_s) -\sigma_s(X_s^{X_0,\nu})]\d W_s\right|^k\\
&\leq C_1\E\left(\int_0^r|X^{X_0,\mu}_s-X_s^{X_0,\nu}|^2\d s\right)^\frac{k}{2}\\
&\leq \frac{1}{2}\E\sup_{t\in[0,r]}|X^{X_0,\mu}_t-X^{X_0,\nu}_t|^k+C_2\E\int_0^r|X^{X_0,\mu}_s-X_s^{X_0,\nu}|^k\d s.
\end{split}\end{equation}
\eqref{dif} together with \eqref{mnu00} and Gronwall's inequality yields
\begin{align*}
\W_{k}(\Phi^{X_0}_t(\mu),\Phi^{X_0}_t(\nu))&\leq(\E\sup_{s\in[0,t]}|X^{X_0,\mu}_s-X_s^{X_0,\nu}|^k)^{\frac{1}{k}}
\leq C_3\int_0^t[\W_k(\mu_s,\nu_s)+\|\mu_s-\nu_s\|_{k,var}]\d s
\end{align*}
for some constant $C_3>0$.
Therefore, for any $\lambda>0$, we have
  \begin{align}\label{Wks}
\sup_{t\in[0,T]}\e^{-\lambda t}\W_{k}(\Phi^{X_0}_t(\mu),\Phi^{X_0}_t(\nu))\leq \frac{C_3}{\lambda }\sup_{t\in[0,T]}\e^{-\lambda t}[\W_k(\mu_t,\nu_t)+\|\mu_t-\nu_t\|_{k,var}].
\end{align}
Next, let
\begin{align*}
\zeta_s&=\sigma_s^{-1}(X_s^{X_0,\mu})(B_s(X_s^{X_0,\mu},\nu_s)-B_s(X_s^{X_0,\mu},\mu_s)), \ \ s\in[0,T],\\
R(t)&=\exp\left\{\int_0^t\<\zeta_s,\d W_s\>-\frac{1}{2}\int_0^t|\zeta_s|^2\d s\right\},\ \ t\in[0,T],\\
W_t^{\mu,\nu}&=W_t-\int_{0}^{t}\zeta_s\d s,\ \ t\in[0,T].
\end{align*}
Then we have
 \begin{equation*} \begin{cases}
\d (X_t^{X_0,\mu})^{(1)}=b_t(X_t^{X_0,\mu})\d t, \\
\d (X_t^{X_0,\mu})^{(2)}=B_t(X_t^{X_0,\mu},\nu_t)\d t+\sigma_t(X_t^{X_0,\mu})\d W^{\mu,\nu}_t.
\end{cases}\end{equation*}
Noting that $|\zeta_s|\leq K\|\sigma^{-1}\|_\infty(\W_k(\mu_s,\nu_s)+\|\mu_s-\nu_s\|_{k,var})$ due to {\bf(C1)}-{\bf(C2)} and \eqref{cod}, Girsanov's theorem yields
 $$\Phi_t^{X_0}(\nu)(f)=\E( R(t)f(X_t^{X_0,\mu})),\ \ f\in\scr B_b(\R^{m+d}), t\in[0,T].$$
Therefore, by the Cauchy-Schwarz inequality for conditional expectation, we obtain
 \begin{align}\label{kva}
\nonumber&\|\Phi_t^{X_0}(\nu)-\Phi_t^{X_0}(\mu)\|_{k,var}\\
&= \E[|R(t)-1|(1+|X_t^{X_0,\mu}|^k)]\\
\nonumber&\leq \E\left([\E(|R(t)-1|^2|\F_0)]^{\frac{1}{2}}[\E((1+|X_t^{X_0,\mu}|^k)^2|\F_0)]^{\frac{1}{2}}\right)
\end{align}
Observe that
 \begin{equation}\begin{split}\label{RRS}
&[\E(|R(t)-1|^2|\F_0)]^{\frac{1}{2}}\\
&\leq \left[\exp\left\{c\int_0^t(\W_k(\mu_s,\nu_s)+\|\mu_s-\nu_s\|_{k,var})^2\d s\right\}-1\right]^{\frac{1}{2}}\\
&\leq \exp\left\{\frac{c}{2}\int_0^t(\W_k(\mu_s,\nu_s)+\|\mu_s-\nu_s\|_{k,var})^2\d s\right\}\\
&\quad\quad\times \sqrt{c}\left(\int_0^t(\W_k(\mu_s,\nu_s)+\|\mu_s-\nu_s\|_{k,var})^2\d s\right)^{\frac{1}{2}}
 \end{split}\end{equation}
 for some constant $c>0$.
For any $N\geq 1$, let
\begin{equation}\label{dds}\scr P_{k,X_0}^{N,T}=\{\mu\in C([0,T],\scr P_k),\mu_0:=\L_{X_0}, \sup_{t\in[0,T]}\e^{-Nt}(1+\mu_t(|\cdot|^k))\leq N\}.
\end{equation}
Then it is clear that as $N\uparrow\infty$,
$$\scr P_{k,X_0}^{N,T}\uparrow\scr P_{k,X_0}^{T}=\{\mu\in C([0,T],\scr P_k),\mu_0=\L_{X_0}\}.$$
So, it remains to prove that there exists a constant $N_0>0$ such that for any $N\geq N_0$, $\Phi^{X_0}$ is a contractive map on $\scr P_{k,X_0}^{N,T}$.

Firstly, it follows from {\bf (C2)} and the BDG inequality that there exists a constant $c_1>0$ such that for any $\mu\in \scr P_{k,X_0}^{N,T}$,
\begin{align*}
\e^{-Nt}\E(1+|X_t^{X_0,\mu}|^k)&\leq \E(1+|Z_0|^k)+c_1\e^{-Nt}\int_0^t\E(1+|Z_s^{Z_0,\mu}|^k)\d s\\
&\qquad\quad+c_1\e^{-Nt}\int_0^t(1+\mu_s(|\cdot|^k))\d s\\
&\leq \E(1+|Z_0|^k)+\frac{c_1}{N}\sup_{s\in[0,t]}\e^{-Ns}\E(1+|Z_s^{Z_0,\mu}|^k)+c_1.
\end{align*}
Aa a result, there exists a constant $N_0>1$ such that for any $N\geq N_0$, $\Phi^{X_0}$ maps $\scr P_{k,X_0}^{N,T}$ to $\scr P_{k,X_0}^{N,T}$.
Next, we derive from \eqref{kva} and \eqref{RRS} that
  \begin{align*}
&\|\Phi_t^{X_0}(\nu)-\Phi_t^{X_0}(\mu)\|_{k,var}\leq C_0(N)\left(\int_0^t(\W_k(\mu_s,\nu_s)+\|\mu_s-\nu_s\|_{k,var})^2\d s\right)^{\frac{1}{2}},\ \ \mu,\nu\in \scr P_{k,X_0}^{N,T}
\end{align*}
for some constant $C_0(N)>0$, which implies that
\beq\label{ESTM'}\beg{split}
\sup_{t\in[0,T]}\e^{-\lambda t}\|\Phi_t^{X_0}(\nu)-\Phi_t^{X_0}(\mu)\|_{k,var}
&\leq \frac{C(N)}{\sqrt{\lambda}}\tilde{\W}_{k,\lambda}(\mu,\nu),\end{split}
 \end{equation}
 here for any $\lambda>0$,
$$\tilde{\W}_{k,\lambda}(\mu,\nu):= \sup_{t\in [0,T]} \e^{-\lambda t}(\|\nu_t-\mu_t\|_{k,var}+\W_k(\mu_t,\nu_t)), \ \ \mu,\nu\in \scr P_{k,X_0}^{T}.$$
Combining \eqref{ESTM'} with \eqref{Wks}, we conclude that for any $N\geq N_0$, there exists a constant $\lambda(N)>0$ such that $\Phi^{X_0}$ is a strictly contractive map on $(\scr P_{k,X_0}^{N,T},\tilde{\W}_{k,\lambda(N)})$. Therefore, the proof is completed by the Banach fixed point theorem and \eqref{dds}.

\end{proof}
\beg{thebibliography}{99}

\bibitem{C} P. E. Chaudru de Raynal, \emph{Strong well-posedness of McKean-Vlasov stochastic differential equation with H\"{o}lder drift,} Stochastic Process. Appl. 130(2020), 79-107.

\bibitem{CN} P.-E. Chaudry De Raynal, N. Frikha, \emph{Well-posedness for some non-linear SDEs and related PDE on the Wasserstein space}, J. Math. Pures Appl. 159(2022), 1-167.




\bibitem{FF} E. Fedrizzi, F. Flandoli, E. Priola, J. Vovelle, \emph{Regularity of stochastic kinetic equations,} Electron. J. Probab. 22(2017), 1-48.



\bibitem{GW} A. Guillin, F.-Y. Wang, \emph{Degenerate Fokker-Planck equations: Bismut formula, gradient estimate and Harnack inequality,} J. Differential Equations 253(2012), 20-40.





\bibitem{HS} X. Huang, Y. Song, \emph{Well-posedness and regularity for distribution dependent SPDEs with singular drifts,} Nonlinear Anal. 203(2021), 112167.



\bibitem{HW18} X. Huang, F.-Y. Wang, \emph{Distribution dependent SDEs with singular coefficients,} Stochastic Process. Appl. 129(2019), 4747-4770.

\bibitem{HWJMAA} X. Huang, F.-Y. Wang, \emph{Singular McKean-Vlasov (reflecting) SDEs with distribution dependent noise,} J. Math. Anal. Appl. 514(2022), 126301 21pp.

\bibitem{HW22a} X. Huang, F.-Y. Wang, \emph{Log-Harnack Inequality and Bismut Formula for Singular McKean-Vlasov SDEs,} arXiv:2207.11536.




\bibitem{M} H. P. McKean, \emph{A class of Markov processes associated with nonlinear parabolic equations,} Proc Natl Acad Sci U S A, 56(1966), 1907-1911.

\bibitem{MV} Yu. S. Mishura, A. Yu. Veretennikov, \emph{Existence and uniqueness theorems for solutions of McKean-Vlasov stochastic equations,} Theor. Probability and Math. Statist. 103(2020), 59-101.

\bibitem{Pin} M. S. Pinsker, \emph{Information and Information Stability of Random Variables and Processes,} Holden-Day, San Francisco, 1964.

\bibitem{RW}  P. Ren, F.-Y. Wang, \emph{Exponential convergence in entropy and Wasserstein for McKean-Vlasov SDEs,}  Nonlinear Anal. 206(2021), 112259.
\bibitem{RZ} M. R\"{o}ckner, X. Zhang, \emph{Well-posedness of distribution dependent SDEs with singular drifts,} Bernoulli 27(2021), 1131-1158.



\bibitem{V} C. Villani,  \emph{Hypocoercivity,} Mem. Amer. Math. Soc. 202(2009).

\bibitem{Wbook} F.-Y. Wang, \emph{Harnack Inequality for Stochastic Partial Differential Equations,} Springer, New York, 2013.

\bibitem{FYW1} F.-Y. Wang, \emph{Distribution-dependent SDEs for Landau type equations,} Stochastic Process. Appl. 128(2018), 595-621.

\bibitem{W2} F.-Y. Wang, \emph{Hypercontractivity and Applications for Stochastic Hamiltonian Systems,} J. Funct. Anal. 272(2017), 5360-5383.

\bibitem{FYW2} F.-Y. Wang, \emph{Distribution dependent reflecting stochastic differential equations,} arXiv:2106.12737.

\bibitem{WZ1} F.-Y. Wang, X. Zhang, \emph{Derivative formula and applications for degenerate diffusion semigroups,} J. Math. Pures Appl. 99(2013),726-740.





\bibitem{Z1} X. Zhang, \emph{Stochastic flows and Bismut formulas for stochastic Hamiltonian systems,} Stochastic Process. Appl. 120(2010), 1929-1949.



\bibitem{Z5} X. Zhang, \emph{Second order McKean-Vlasov SDEs and kinetic Fokker-Planck-Kolmogorov equations,} arXiv:2109.01273.

\end{thebibliography}

\end{document}